\documentclass[a4paper,11pt]{amsart}

\usepackage[utf8]{inputenc}
\usepackage[english]{babel}
\usepackage{mathrsfs}
\usepackage{a4wide}
\usepackage{amsmath, amssymb}
\usepackage[pdfborder={1 1 0.3},bookmarks=false]{hyperref}

\newtheorem{theorem}{Theorem}[section]

\newtheorem{lemma}[theorem]{Lemma}
\newtheorem{proposition}[theorem]{Proposition}
\newtheorem{corollary}[theorem]{Corollary}

\theoremstyle{definition}

\newtheorem{definition}[theorem]{Definition}

\DeclareMathOperator{\Aut}{Aut}
\DeclareMathOperator{\Ext}{Ext}
\DeclareMathOperator{\Char}{char}
\DeclareMathOperator{\Gal}{Gal}
\newcommand\Z{\ensuremath{\mathbf{Z}}}
\newcommand\N{\ensuremath{\mathbf{N}}}

\newcommand\Q{\ensuremath{\mathbf{Q}}}
\newcommand\F{\ensuremath{\mathbf{F}}}
\newcommand{\calA}{\mathcal{A}}
\newcommand{\frakP}{\mathfrak{P}}
\newcommand{\pibar}{{\bar{\pi}}}
\newcommand{\Alg}[1]{\overline{#1}}
\newcommand{\bad}{\pi}

\newcommand{\gsc}{J}
\newcommand{\catc}{\mathscr{C}}
\newcommand{\maxextcatc}{T_\catc}

\begin{document}
\title{\'Etale subquotients of prime torsion of abelian schemes}
\author{Hendrik Verhoek}
\begin{abstract}
Let $A$ be an abelian variety over a number field 
$K$ with good reduction outside a finite set of primes $S$.
We show that if the $\ell$-torsion subgroup schemes $A[\ell^n]$
lie in a certain category of group schemes,
then $A[\ell^n]$ does not contain any subgroup schemes that are \'etale
or are of multiplicative type.
\end{abstract}
\maketitle
\tableofcontents

\section{Introduction} \label{sec:intro}

Let $K$ be a number field with ring of integers $O_{K}$ and let $S$ be a finite set of primes in $O_K$.
Denote by $O_S$ the ring of $S$-integers of $K$.
Let $\ell$ be a rational prime such that none of the primes in $S$ divides $\ell$.

\begin{definition}
Let $\catc$ be a subcategory of the category of finite flat commutative group schemes over $O_S$
of $\ell$-power order, such that $\catc$ is closed under taking products, subquotients
and Cartier duality.
\end{definition}

In addition, with an eye towards our main theorem stated below, 
we state the following two conditions that the category $\catc$ might or might not satisfy.
These conditions involve simple group schemes in $\catc$, 
i.e., group schemes that have no non-trivial closed flat subgroup schemes.

\vskip 10pt
\noindent
{\bf Condition $(1)$:} For all simple non-\'etale group schemes $T$ in $\catc$ 
and all simple \'etale group schemes $E$ in $\catc$,
the group $\Ext_\catc^{1}(T,E)$ is trivial.

\vskip 10pt
\noindent
{\bf Condition $(2)$:}
Let $F$ be the compositum of all $K(E)$, where $E$ runs over all simple \'etale group schemes $E$ in $\catc$.
Then the extension $F/K$ is finite and the maximal abelian extension $R$ of $F$,
that is unramified outside $S$ and at most tamely
ramified at primes over $S$, is a cyclic extension.

\vskip 10pt

Let $A$ be an abelian variety over $K$ with good reduction outside $S$,
let $\calA$ denote its N\'eron model.
Denote by $\calA[\ell^n]$ the $\ell^n$-torsion subgroup scheme of $\calA$.
The schemes $\calA[\ell^n]$ are finite flat commutative group scheme over $O_S$. 
We prove:

\begin{theorem} \label{thm:abvar_torsionfilter}
Let $A$ be an abelian variety such that $\calA[\ell^n]$ is an object
in $\catc$ for all $n \in \N$.
If Conditions $(1)$ and $(2)$ hold for the category $\catc$,
then $\calA[\ell]$ does not have subquotients that are \'etale or 
of multiplicative type.
\end{theorem}
As an application, we prove:

\begin{corollary}
There do not exist abelian varieties over $\Q(\sqrt{13})$
and $\Q(\sqrt{17})$ with good reduction everywhere.
\end{corollary}

In the rest of the article we continue as follows.
First we indicate how one finds simple group schemes in $\catc$.
Then we discuss filtrations and extensions of group schemes in $\catc$ and prove 
Theorem \ref{thm:abvar_torsionfilter}.
The proof is divided into three steps, the same steps that 
can be found in \cite{Fontaine:1985}, \cite{Schoof:2003} and \cite{Schoof:2005}
and that prove the non-existence or unique up to isogeny results of abelian varieties
with good or semi-stable reduction at the primes in $S$.

\begin{enumerate}
\item Define a category $\catc$ that contains $\calA[\ell^n]$ for all $n$

\item Find the simple objects in the category $\catc$ by 
using the generic fiber of objects in $\catc$ annihilated by $\ell$ 
and the discriminant bounds of Odlyzko to classify the generic fibers of simple objects in $\catc$,
and subsequently use theorems of Oort-Tate and Raynaud to determine the simple objects up to isomorphism.
Verify that Condition $(2)$ holds.

\item Calculate various extension groups of the objects in $\catc$ and 
verify Condition $(1)$. 
If both conditions hold, then apply Theorem \ref{thm:abvar_torsionfilter}.
\end{enumerate}

\section{The generic fiber of simple group schemes}

The generic fiber of a finite flat commutative group scheme $\gsc$ over $O_S$ is a group scheme over $K$,
which we denote by $\gsc_K$.
Since $\Char(K) = 0$, $\gsc_K$ is an \'etale group scheme.
Therefore, the group scheme $\gsc_K$ is just an abelian group $\gsc(\Alg{K})$ together with the Galois action 
$\rho_{\gsc}: G_{K} \longrightarrow \Aut(\gsc(\Alg{K}))$.
We denote by $K(\gsc)$ the field extension obtained by adjoining the $\Alg{K}$-points of $\gsc$ to $K$.
The representation $\rho_\gsc$ factors through a finite Galois extension $K(\gsc)/K$.
By considering the generic fiber $\gsc_K$ we obtain not only information about the group scheme $\gsc$ considered over $K$,
but also as a scheme over $O_S$. 
It is even true that, under certain conditions (see \cite{Raynaud:1974}), 
the generic fiber uniquely determines the group scheme $\gsc$ over $O_S$.

A first step to understand the category $\catc$ is to classify its simple objects
up to isomorphism.
Every simple object is annihilated by $\ell$:
if not, the Zariski closure of the $\ell$-torsion points in the generic fiber would form a non-trivial
closed flat subgroup scheme.
Since by assumption $\catc$ is closed under taking subquotients,
this subgroup scheme would again be in $\catc$.

Define $\maxextcatc$ to be the compositum of 
all fields $K(\gsc)$, where the $\gsc$ are group schemes in $\catc$ that are annihilated by $\ell$.
We call $\maxextcatc$ the \emph{maximal $\ell$-torsion extension} of $\catc$.
This extension $\maxextcatc$ need not be finite in general.
The reason that we are interested in the maximal $\ell$-torsion extension of $\catc$
is that if $\maxextcatc$ is finite, 
it enables us to find the simple objects in $\catc$.
Namely, the $\Alg{K}$-points of every simple object generate an extension
that is a subfield of the maximal $\ell$-torsion extension of $\catc$.
As a side note we mention that to find $\maxextcatc$ in practice,
it is helpful that the category $\catc$ is closed under taking products.

\begin{lemma} \label{lem:irrgenfiber}
If $\gsc$ is a simple finite flat commutative group scheme over $O_S$, then 
the representation $\rho_{\gsc} : G_{K} \rightarrow \Aut(\gsc(\Alg{K}))$
is irreducible.
\end{lemma}
\begin{proof}
Suppose $\rho_{\gsc}$ admits a non-trivial $G_{K}$-stable subgroup $V$.
Since the closure of the generic point of $O_S$ is $O_S$ (recall that $O_S$ is a Dedekind ring),
taking the Zariski closure of $V$ gives a non-trivial closed flat subgroup scheme of $\gsc$.
This closure is equal to $\gsc$ because $\gsc$ is simple.
The generic fiber of the closure, which is equal to $J(\Alg{K})$, is contained in $V$.
\end{proof}

The generic fiber of a simple object $\gsc$ in $\catc$ is a
simple $\F_\ell[\Gal(K(\gsc)/K)]$-module.
Since simple objects are killed by $\ell$,
such a generic fiber is also a simple $\F_{\ell}[\Gal(\maxextcatc/K)]$-module.
Therefore we classify all simple $\F_{\ell}[\Gal(\maxextcatc/K)]$-modules.
If we can find a relatively large normal $\ell$-subgroup $H$ in $\Gal(\maxextcatc/K)$, 
it is easier to classify irreducible submodules:
the representation $\rho_{\gsc}$ factors not only through $\Gal(L/K)$,
but also through the quotient of $\Gal(L/K)$ by $H$.
This is an immediate consequence of:

\begin{lemma} \label{lem:factor_through_p_group}
Let $\gsc$ be a simple object in $\catc$.
Then $\Gal(K(\gsc)/K)$ contains no non-trivial normal $\ell$-subgroup.
\end{lemma}
\begin{proof}
The representation $\rho_{\gsc}$ factors through $\Gal(K(\gsc)/K)$.
Let $H$ be a non-trivial normal $\ell$-subgroup of $\Gal(K(\Alg{\gsc})/K)$.
Then $H$ must act faithfully as a $\ell$-group on the $\ell$-group $\gsc(\Alg{K})$,
but this is impossible. There are non-trivial fixed points of $\gsc(\Alg{K})$ under this action and they form a closed flat subgroup scheme
of $\gsc$, which must equal $\gsc$ since $\gsc$ is simple. 
\end{proof}

Finally, once simple $\F_{\ell}[\Gal(\maxextcatc/K)]$-modules have been found, the question
remains if they extend to finite flat commutative group schemes over $O_S$.
This is addressed in the work of Raynaud \cite{Raynaud:1974} and Oort-Tate \cite{TateOort:1970}.

\section{Filtrations by simple group schemes} \label{sec:filtrations}

In this section we discuss filtrations of group schemes in $\catc$ by simple subgroup schemes.
These filtrations will be used to prove Theorem \ref{thm:abvar_torsionfilter}.
Each finite flat commutative group scheme $\gsc$ contains a simple closed flat subgroup scheme $\gsc'$.
The same is true for $\gsc/\gsc'$.
Continuing like this we obtain a filtration of $\gsc$:

\begin{definition}
A \emph{(left) filtration of a finite flat commutative group scheme $\gsc$}
is an ordered set $\{ \gsc_i \}_{i=1}^n$ such that
\begin{itemize}
\item $\gsc_1$ is a simple closed flat subgroup scheme of $F_1 := \gsc$

\item for $1 < i < n$, let $\gsc_i$ be a simple closed flat subgroup scheme of 
$F_i := F_{i-1}/\gsc_{i-1}$

\item $\gsc_n$ is simple
\end{itemize}
We call $n$ the length of the filtration.
\end{definition}

We note that by using Cartier duality, we can get another (right) filtration.
If $A$ is a simple group scheme occurring in a filtration (or equivalently all filtrations) of $\gsc$,
we say that \emph{$\gsc$ admits $A$}.

\begin{lemma} \label{lem:filter_force_subgroupscheme}
Let $\gsc$ be a group scheme in $\catc$ that admits the simple group scheme $A$.
Suppose that for each simple $B$ with $B \not \simeq A$ occurring in the filtration of $\gsc$,
the group $\Ext_\catc^{1}(A,B)$ is trivial.
Then $A$ is a closed flat subgroup scheme of $\gsc$.
\end{lemma}
\begin{proof}
Consider the short exact sequence
\begin{align} \label{eqn:lem_filter_force}
0 \longrightarrow \gsc' \longrightarrow \gsc \longrightarrow \gsc/\gsc' = F_2 \longrightarrow 0 ,
\end{align}
where $\gsc'$ is simple.
If $\gsc' \simeq A$ there is nothing to prove, so assume $A \not \simeq \gsc'$.
We proceed by induction on the length of the filtration of $\gsc$.
The statement of the lemma holds for length one and two.
By induction we have the following exact sequence:
$$
0 \longrightarrow A \longrightarrow \gsc/\gsc' = F_2 \longrightarrow F_3 \longrightarrow 0 .
$$
The pull-back of $A$ by $\gsc$ over $F_2$, using (\ref{eqn:lem_filter_force}), gives the short exact sequence
$$
0 \longrightarrow \gsc' \longrightarrow \gsc \times_{F_2}  A \longrightarrow A \longrightarrow 0 .
$$
The group scheme $\gsc \times_{F_2} A$ is a closed flat subgroup scheme of $\gsc$.
By hypothesis, $\gsc \times_{F_2}  A \simeq A \times \gsc'$.
Hence $A$ is a closed flat subgroup scheme of $\gsc$.
\end{proof}

\begin{corollary} \label{cor:filter1}
Let $\gsc$ be a finite flat commutative group scheme in the category $\catc$ that
admits a simple group scheme $A$.
If for each simple $B$ with $B \not \simeq A$ occurring in the filtration of $\gsc$,
the group $\Ext_\catc^{1}(A,B)$ is trivial,
then there exists a closed flat subgroup scheme $\gsc'$ of $\gsc$ 
admitting only copies of $A$ and such that $\gsc/\gsc'$ 
does not admit $A$.
\end{corollary}
\begin{proof}
We proceed by induction on the length of the filtration of $\gsc$.
If the length of $\gsc$ is one, we are done.
If the length is two, we are again done by hypothesis.
Suppose the length of $\gsc$ is $k$ and the statement holds if the length is at most $k-1$.
By Lemma \ref{lem:filter_force_subgroupscheme} we can write
$0 \subset A \subset \gsc$.
By induction, there exists a closed flat subgroup scheme $\gsc''$ of $\gsc/A$ 
such that $(\gsc/A)/\gsc''$ does not admit $A$ and $\gsc''$ only admits copies of $A$.
Then $\gsc' := \gsc \times_{\gsc/A} \gsc''$ verifies the condition of the statement.
\end{proof}

The next proposition 
resembles the fact that for finite flat commutative group schemes over a local henselian ring,
the quotient by the connected component is an \'etale group scheme.
See for instance \cite[p. 138]{CornellSilvermanStevens:1997}.

\begin{proposition} \label{prop:filter1}
If Condition $(1)$ holds for the category $\catc$,
then for any $\gsc$ in $\catc$ we have an exact sequence
$$
0 \longrightarrow \gsc' \longrightarrow \gsc \longrightarrow \gsc'' \longrightarrow 0 
$$
such that $\gsc''$ is \'etale and $\gsc'$ does not admit an \'etale scheme.
\end{proposition}
\begin{proof}
Let $\gsc^*$ be the Cartier dual of $\gsc$.
It suffices to show that $\gsc^*$ contains a subgroup scheme $M$ of multiplicative type 
such that $\gsc^*/M$ does not admit a simple group scheme of multiplicative type.
We may suppose that $\gsc^*$ admits a simple group scheme of multiplicative type;
if not, we are done.
Then by Lemma \ref{lem:filter_force_subgroupscheme}, the group scheme $\gsc^*$
has a simple subgroup scheme of multiplicative type.

Next, suppose that $\gsc^*$ has a subgroup of multiplicative type $M'$ such
that $\gsc^*/M'$ admits a simple group scheme of multiplicative type;
if not, we are done again.
Then again by Lemma \ref{lem:filter_force_subgroupscheme}, the group scheme
$\gsc^*/M'$ has a simple subgroup scheme of multiplicative type $M''$.
Now $M'' \times_{\gsc^*/M'} \gsc^*$ is a closed flat subgroup
scheme of $\gsc^*$ and sits inside the short exact sequence
$$
0 \longrightarrow M' \longrightarrow  \gsc^* \times_{\gsc^*/M'} M'' \longrightarrow M'' \longrightarrow 0 .
$$  
Hence $M'' \times_{\gsc^*/M'} \gsc^*$ is an extension of two group schemes of multiplicative type
and therefore itself of multiplicative type.
Proceeding this way, we find a subgroup scheme of multiplicative type
$M$ such that $\gsc^*/M$ does not admit 
a simple group scheme of multiplicative type.
\end{proof}

\section{Application to abelian varieties} \label{sec:filt_ab_var}

In this section, we will prove Theorem \ref{thm:abvar_torsionfilter}.
We first state two auxiliary lemmas:

\begin{lemma} \label{lem:cyclicpgroup}
Let $p$ be a prime and $G$ be a finite
$p$-group such that $G/[G,G]$ is cyclic.
Then $G$ is cyclic.
\end{lemma}
\begin{proof}
The Frattini subgroup $\text{Frat}(G)$ of $G$ is equal to $[G,G] G^{p}$.
The group $G/\text{Frat}(G)$ is by hypothesis a cyclic group of order $p$.
Burnside's basis Theorem \cite[Theorem 12.2.1, p. 176]{Hall:1959} implies that $G$ is cyclic.
\end{proof}

\begin{lemma} \label{lem:mod_ext_gal}
Let $G$ be a group and $A,B,C$ be finite $G$-modules
such that $A$ and $C$ have trivial $G$-action and $G$ acts faithfully on $B$.
Let
$$
0 \longrightarrow A \longrightarrow B \longrightarrow C \longrightarrow 0
$$
be an exact sequence of $G$-modules.
Let $k$ denote the number of generators of $C$.
Then $\# G$ divides $(\# A)^{k}$.
\end{lemma}
\begin{proof}
We leave the proof to the reader.
\end{proof}

\begin{lemma} \label{prop:etale_field}
Let $R$ be the field as before in Condition $(2)$.
If Condition $(2)$ holds for the category $\catc$, then any \'etale object $\gsc$ in $\catc$
becomes constant over $R$.
\end{lemma}
\begin{proof}
Let $\gsc$ be any \'etale group scheme in $\catc$.
We claim that $\Lambda=\Gal(F(\gsc_F)/F)$ is an $\ell$-group.
The proof proceeds by induction on the order of $\gsc$.
There exists an \'etale group scheme $\gsc'$ in $\catc$ such that
we have the following short exact sequence of group schemes over the field $F$:
$$
0 \longrightarrow \gsc'_F \longrightarrow \gsc_F \longrightarrow \Z/\ell\Z \longrightarrow 0 .
$$
By induction, $\Gal(F(\gsc'_F)/F)$ is an $\ell$-group.
Apply Lemma \ref{lem:mod_ext_gal} to finish the induction and prove the claim.

We note that $\Lambda/[\Lambda,\Lambda]$ is an abelian $\ell$-group and hence the fixed field of $[\Lambda,\Lambda]$
is at most tamely ramified at primes dividing the primes in $S$.
This fixed field is contained in $R$,
which by assumption is a cyclic extension of $F$.
Hence also $\Lambda/[\Lambda,\Lambda]$ is cyclic and by Lemma \ref{lem:cyclicpgroup} the group $\Lambda$ is cyclic.
We conclude that $F(\gsc_F)$ is contained in $R$, which is exactly what we wanted to prove.
\end{proof}

For example, if the 
Hilbert class field of $F$ is trivial and $S$ contains only one prime that does not split in $F/K$, 
then $R$ is a cyclic extension of $F$.

\begin{proposition} \label{prop:points}
Let $q \notin S$ be a prime in $O_{K}$
that is inert in $R/K$.
Suppose that Conditions $(1)$ and $(2)$ hold for the category $\catc$.
Then for any $\gsc$ in $\catc$ having $n$ simple \'etale group schemes 
and $m$ simple group schemes of multiplicative type in its filtration,
the following inequalities hold:
$$
| \gsc_q ( \F_q ) |  \geq \ell^{n} \quad \text{and} \quad | \gsc^*_q (\F_q) | \geq \ell^{m} .
$$
\end{proposition}
\begin{proof}
Let $R$ as before.
By Proposition \ref{prop:etale_field}
all \'etale objects in $\catc$ become constant over $R$.
Let $E$ be the \'etale quotient of $\gsc$ as in Proposition \ref{prop:filter1}.
Let $\frakP$ be a prime in $O_R$ lying above $q$.
The residue field $\F_\frakP$ is equal to $\F_q$.
Since $E_R$ is constant,
it follows that also $E_{\frakP}$ is constant and hence that $\gsc_R$ has at
least $\ell^{n}$ points in the fiber at ${\frakP}$.
The inequality $| \gsc_q ( \F_q ) |  \geq \ell^n$ follows.
The second inequality follows by Cartier duality.
\end{proof}

We are now able to prove Theorem \ref{thm:abvar_torsionfilter}:

\begin{proof}
By contradiction, suppose that $A[\ell]$ contains $k$ simple \'etale subquotients.
Then for any prime $q$ that is not in $S$ and is inert in $R/K$,
Proposition \ref{prop:points} says that the number of $\ell$-torsion
points of $A$ in the fiber at $q$ is at least $\ell^{k}$.
Hence $A[\ell^n]$ has at least $\ell^{kn}$ points in the fiber at $q$.
This is in contradiction with the fact that $A(\F_q)$ is finite for $n$ sufficiently large.
Let $A^{\text{dual}}$ be the dual abelian variety of $A$.
For each $n$, the group scheme $A^{\text{dual}}[\ell^n]$ is the Cartier dual of $A[\ell^n]$.
If $A[\ell]$ has subquotients of multiplicative type,
then $A^{\text{dual}}[\ell]$ has \'etale subquotients which is impossible by the same argument
given above but now applied to the abelian variety $A^{\text{dual}}$.
\end{proof}

We apply Theorem \ref{thm:abvar_torsionfilter} together with the three steps described in the introduction
to prove:

\begin{theorem}
There are no non-zero abelian varieties over $\Q(\sqrt{13})$ with good reduction everywhere.
\end{theorem}
\begin{proof}
We follow the steps mentioned in the introduction: 

\begin{enumerate}
\item
We define $\catc$ to be the category of finite flat commutative group schemes of $2$-power order
over $O=\Z[\frac{1+\sqrt{13}}{2}]$.

\item 
By \cite{Fontaine:1985} we know that the root discriminant $\delta$ of the extension $\maxextcatc/\Q$
satisfies $\delta < 4\sqrt{13}$.
By Odlyzko's tables this implies that $[\maxextcatc:\Q] < 60$.
Group schemes in $\catc$ annihilated by $2$ are isomorphic to $\mu_2, \Z/2\Z$
and the non-trivial extensions of $\Z/2\Z$ by $\mu_2$ described
in \cite[Section 8.7, p.251]{KatzMazur:1985} using the units $-1$ and $\eta = \frac{3+\sqrt{13}}{2}$.
Hence $\maxextcatc$ contains $i$ and the square root of $\eta$:
$$
\Q(\sqrt{13}) \subset_{4} \Q(i, \sqrt{\eta}) \subset_{\leq 7} \maxextcatc .
$$
The extension $\maxextcatc/\Q(i, \sqrt{\eta})$ is unramified outside $2$
and is solvable. However, the smallest non-trivial abelian extension unramified outside $2$
of $\Q(i, \sqrt{\eta})$ is a subfield of the ray class field of conductor $\bad_2^6$,
where $\bad_2$ is the unique prime above $2$ in $\Q(i, \sqrt{\eta})$.
This subfield violates the root discriminant bound on $\maxextcatc$.
It follows that $\maxextcatc= \Q(i, \sqrt{\eta})$.
By Lemma \ref{lem:factor_through_p_group} this implies that every simple object in $\catc$ has rank $2$.

Since $2$ is inert in $\Q(\sqrt{13})$, this implies by \cite[Corollary, p.21]{TateOort:1970} that the simple group schemes in $\catc$
are $\mu_{2}$ and $\Z / 2\Z$.
One now checks that Condition $(2)$ is satisfied.

\item
For this category, 
Now we use Theorem \cite[Prop. 2.6]{Schoof:2003} to verify that Condition $(1)$ is satisfied.

The $2$-torsion of a non-zero abelian variety over $\Q(\sqrt{13})$ with good reduction everywhere is an object in $\catc$,
and this $2$-torsion subgroup scheme must be filtered by copies of $\mu_2$ or $\Z/2\Z$.
This, however, contradicts Theorem \ref{thm:abvar_torsionfilter}.

\end{enumerate}
\end{proof}

As another example, we show that:

\begin{theorem}
There are no non-zero abelian varieties over $\Q(\sqrt{17})$ with good reduction everywhere.
\end{theorem}
\begin{proof}
We follow the steps mentioned in the introduction: 
\begin{enumerate}
\item

Let $\catc$ be the category of finite flat commutative group schemes of $2$-power order
over $O=\Z[\frac{1 + \sqrt{17}}{2}]$.
We will see that the category $\catc$ does not satisfy Condition $(1)$ of Theorem \ref{thm:abvar_torsionfilter}.

\item
We find the maximal $2$-torsion extension $\maxextcatc/\Q(\sqrt{17})$ of $\catc$.
We leave it as an exercise to show that the extension $\maxextcatc/\Q(\sqrt{17})$ is finite and has degree a power of $2$.
So we can apply Lemma \ref{lem:factor_through_p_group}.
By factoring $2=\pi \bar{\pi}$ in $O$ we find the following simple group schemes:
$\mu_{2}, \Z / 2\Z , G_{\pi}$ and $G_{\bar{\pi}}$, where we refer to 
\cite{TateOort:1970} for the meaning of $G_{\pi}$ and $G_{\bar{\pi}}$.

\item
The only simple \'etale group scheme is $\Z/2\Z$ and we immediately verify 
Condition $(2)$ for our category $\catc$.
However, Condition $(1)$ fails because $\Ext^1_{O}(\mu_2,\Z/2\Z)$ is non-trivial
due to the splitting of the prime $2$ in $\Q(\sqrt{17})/\Q$:
A non-trivial extension is given by $G_\pi \times G_\pibar$.

Even though Condition $(1)$ does not hold, it is true that all extensions of simple
non-\'etale group schemes by simple \'etale group schemes are annihilated by $2$:
they are products of $G_\pi$'s and $G_{\bar{\pi}}$'s.
Using this, for any abelian variety $A$ over $\Q(\sqrt{17})$ with good reduction everywhere
one deduces that the rank of $A[2^n]$ (which is an object in $\catc$) cannot depend on $n$. Hence there
are no such non-zero abelian varieties.
\end{enumerate}
\end{proof}

We end this article by asking for which square-free integers $D$ do there exist
abelian varieties over $\Q(\sqrt{D})$ with good reduction everywhere?

\bibliography{index}
\bibliographystyle{alpha}
\end{document}